\documentclass[10pt]{article}

 \usepackage{amsfonts,amssymb,graphicx,amsmath, amsthm}
\usepackage[colorlinks=true,linkcolor=blue,citecolor=blue]{hyperref}

\usepackage{multicol}

\newtheorem{rema}{Remark}
\newtheorem{defi}{Definition}
\newtheorem{lemm}{Lemma}
\newtheorem{theo}{Theorem}
\newtheorem{coro}{Corollary}

\newcommand{\C}[1][]{\ensuremath{{\mathbb{C}^{#1}} }}
\newcommand{\R}[1][]{\ensuremath{{\mathbb{R}^{#1}} }}
\renewcommand{\S}[1][]{\ensuremath{{\mathbb{S}^{#1}} }}

\newcommand{\X}[1][]{\ensuremath{{\mathbb{Q}^{#1}} }}
\newcommand{\D}[1][]{\ensuremath{{\mathbb{D}^{#1}} }}

\def\Re{ \mathrm{Re}\, }
\def\Im{ \mathrm{Im}\, }

\renewcommand{\j}{\ensuremath{J}}
\renewcommand{\L}{{\cal L}}

\newcommand{\M}{{\cal M}}

\newcommand{\s}{{\cal S}}
\newcommand{\<}{\langle}
\renewcommand{\>}{\rangle}
\newcommand{\ga}{\gamma}
\newcommand{\pa}{\partial}

\newcommand{\al}{\alpha}
\newcommand{\eps}{\epsilon}
\newcommand{\te}{\theta}
\newcommand{\ka}{\kappa}
\newcommand{\la}{\lambda}
\newcommand{\be}{\beta}
\newcommand{\si}{\sigma}
\newcommand{\ove}{\overline}
\date{}

\title{Lagrangian submanifolds in para-complex Euclidean space}
\author{ Henri Anciaux\footnote{Universit\'e Libre de Bruxelles}, Maikel Samuays\footnote{Universidade de S\~ao Paulo; supported by  FAPESP (2012/02724-3)}}

\begin{document}

\maketitle

\centerline{\textbf {\large{Abstract}}}

\bigskip

{\small 

We address the study of some curvature equations for distinguished submanifolds in para-K\"ahler geometry. We first observe that a para-complex submanifold of a para-K\"ahler manifold is minimal. Next we describe the extrinsic geometry of Lagrangian submanifolds in the para-complex Euclidean space $\D^n$ and discuss a number of examples, such as graphs and normal bundles. We also
characterize Lagrangian surfaces of $\D^2$ which are minimal and have indefinite metric. Finally we describe  those Lagrangian self-similar solutions of the Mean Curvature Flow which are $SO(n)$-equivariant. }

\bigskip

\centerline{\small \em 2010 MSC: 53A10, 53D12
\em }

\section*{Introduction} \
Symplectic manifolds enjoy a distinguished class of submanifolds, namely \em Lagrangian submanifolds. \em They are defined as those submanifolds of maximal dimension (half the dimension of the ambient space) such that the symplectic form vanish on it. In the K\"ahler case, it is interesting to study the metric properties of Lagrangian submanifolds. In particular, Lagrangian submanifolds which are in addition \em minimal, \em i.e., critical points of the volume functional attached to the metric, enjoy  interesting  properties. For example, in complex Euclidean space $\C^n$ (or more generally in a \em Calabi-Yau \em manifold), a locally defined angle function, the \em Lagrangian angle, \em is attached to any Lagrangian submanifold, and minimal ones are characterized by the constancy of their Lagrangian angle. Recently, the study of Lagrangian submanifolds in pseudo-Riemannian K\"ahler manifolds has been addressed (see \cite{MyB},\cite{An3}). 

\medskip

Para-complex geometry, an alternative to complex geometry, is the study of manifolds $\M$ endowed with a $(1,1)$-tensor $J$ satisfying $J^2=Id$ (instead of the usual relation $J^2=-Id$ characterizing complex geometry), satisfying in addition the rank condition $dim (Ker(J - Id)) = dim(Ker(J+Id))=\displaystyle\frac{1}{2} dim \, \M .$ The model space of para-complex geometry is the Cartesian product $\D^n$, where $\D$ is the module of para-complex numbers (see next section for the precise definition). Using the identification $\D^n \simeq \R^{2n} \simeq T^* \R$, we may also consider a natural symplectic structure $\omega$ on $\D^n$. The pair $(J, \omega)$ then defines a pseudo-riemannian metric by the relation $\<\cdot,\cdot\>:=\omega (\cdot, J\cdot)$, which happens to have neutral signature and makes the para-complex structure $J$ compatible in a suitable sense (see next section for more detail).

\medskip

This paper is devoted to the study of two classes of submanifolds 
 that appear naturally in para-K\"ahler geometry, namely para-complex and Lagrangian submanifolds. We first prove that para-complex submanifolds are minimal (like complex submanifolds in  K\"ahler geometry), but unstable (unlike complex submanifolds). Next, we describe the extrinsic geometry of Lagrangian submanifolds and define their Lagrangian angle, whose constancy is equivalent to minimality (like in the K\"ahler case). This is related to the fact, observed in \cite{Me} (see also \cite{HL2}), that minimal Lagrangian submanifolds of $\D^n$ enjoy a kind of ``Lagrangian calibration" and are therefore extremizers of the volume in their Lagrangian isotopy class (but not in their whole isotopy class). Next, we discuss a number of examples of Lagrangian submanifolds, such as minimal Lagrangian graphs, minimal normal bundles. We also 
characterize Lagrangian surfaces of $\D^2$ which are minimal and have indefinite metric. Finally we describe  the Lagrangian self-similar solutions of the Mean Curvature Flow which are $SO(n)$-equivariant.

\section{Preliminaries}
\subsection{The space $\D^n$} \label{Dn} \
The set $\D$ of \em  para-complex \em (or \em split-complex, \em or \em double\em)  numbers
is the two-dimensional real vector space $\R^2$ endowed with the commutative algebra structure whose product  rule is given by
$$ (x , y)  (x' , y')= (xx'+yy', xy'+x'y).$$
 The number $(0,1)$, whose square is $(1,0)$ and not $(-1,0)$, will be denoted by $\tau.$
It is convenient to use the following notation: $(x,y) \simeq z =x + \tau y .$
In particular, one has the same conjugation operator than in $\C$:
$$ \overline{x+ \tau y } = x - \tau y.$$
Since $ z  \bar{z} = x^2 -y^2,$ it is only natural  to introduce  the Minkowski metric $\<\cdot,\cdot\>=dx^2 - dy^2,$ so that  the squared norm $\<z,z\>$ of $z$ is $ z  \bar{z}$.

We also introduce the \em polar coordinates \em as follows: the \em radius \em of $z \in \D$ is 
$r :=|\<z ,z\>|^{1/2}$. If $z$ is non-null, its \em argument \em $\arg z$ is  defined to be the unique real number $\te$ such that 
$$z =\pm \tau^q r  e^{\tau \te} =\pm \tau^q r (\cosh (\te) + \tau \sinh(\te)), \quad q \in \{0,1\}.$$
Next we define the ``para-Cauchy-Riemann" operators on $\D$ by
$$\frac{\pa }{\pa z}:= \frac{1}{2}\left( \frac{\pa }{\pa x} + \tau \frac{\pa }{\pa y} \right) \quad \quad
\frac{\pa }{\pa \bar{z}}:= \frac{1}{2}\left(\frac{\pa }{\pa x} - \tau \frac{\pa }{\pa y} \right).$$
A map $f$ defined on a domain of $\D$ is said to be \em para-holomorphic \em if it satisfies
$\displaystyle\frac{\pa f}{\pa \bar{z}}=0.$
Observe that this is a hyperbolic equation, so a para-holomorphic map needs not to be analytic, and may be merely continuously differentiable.

\medskip

On the Cartesian $n$-product $\D^n$ with para-complex coordinates $(z_1, \ldots ,z_n)$, we define the canonical  para-K\"ahler structure $(J,\<\cdot,\cdot\>)$ by
$$ J(z_1, \ldots,z_n):=(\tau z_1, \ldots ,\tau z_n)$$
and
$$ \<\cdot,\cdot\>:=\sum_{j=1}^{n} dz_j  d\bar{z}_j = \sum_{j=1}^{n}  dx_j^2-dy_j^2.$$
We also introduce the ``para-Hermitian" form:
$$\<\<\cdot,\cdot\>\> := \sum_{j=1}^{n} dz_j \otimes d\bar{z}_j= \sum_{j=1}^{n}  (dx_j^2-dy_j^2) - \tau  \sum_{j=1}^{n} dx_j \wedge dy_j.$$
In other words we recover $\<\cdot,\cdot\>$ by taking the real part of  $\<\<\cdot,\cdot\>\>$. On the other hand,  its imaginary part is (up to sign) the canonical  symplectic form $\omega$ of $T^* \R^n$ under the natural identification $\D^n \simeq T^* \R^n$. Finally, the three structures (metric, para-complex and symplectic) are related by the formula:
$\omega:=\<\cdot,J\cdot\>.$

\subsection{Para-complex and para-K\"ahler manifolds} \
Let $\M$ be a differentiable manifold. 
An \em almost para-complex structure \em on  $\M$ is a $(1,1)$-tensor $J$ satisfying $J^2=Id$ and such that the  two eigendistributions $Ker(J\mp Id)$ have the same rank. On the other hand, a \em para-complex structure \em is an atlas  on $\M$ whose transition maps are local  bi-para-holomorphic diffeomorphisms of  $\D^n$. Exactly as in the complex case, a para-complex atlas defines an almost-complex structure $J$ by the formula $JX = d\varphi^{-1} \big(  \tilde{J} d\varphi (X)\big)$, where we denoted by $\tilde{J}$ the para-complex structure of $\D^n$ and by $\varphi$  a local chart on $\M$.
The question of under which condition the converse is true, hence  the para-complex version of  Newlander-Nirenberg theorem, is the content of the following

\begin{theo}
The almost para-complex structure $J$ comes from a para-complex if and only if its \em para-Nijenhuis tensor \em $N^J$, defined by $$ N^J(X,Y):=[X,Y]+[JX,JY]-J[JX,Y]-J[X,JY]$$
vanishes.
\end{theo}

The proof of this theorem, a consequence of Frobenius theorem, is much simplier than in the complex case. For seek of completness we state it in  Appendix.

Unlike the case of complex manifolds, a para-complex manifold is not necessarily orientable: for example, it is easy to equip the Klein bottle with a para-complex structure: let $\M = \R^2 / \! \!\sim$
where $\sim$ is the equivalence relation defined by $ (u,v) \sim (u+1,v) \sim (1-u, v+1)$. Then the para-complex structure defined  on $\R^2$ by $ J \pa_u= \pa_u$ and $J \pa_v = -\pa_v$ descends to $\M$.

A pair $(J, g)$,  where $J$ is para-complex structure, and $g$  a compatible pseudo-Riemannian metric, is said to be \em a K\"ahler structure \em on $\M$ if the alternated 2-from 
$\omega:=g(J\cdot,\cdot)$ is symplectic, i.e.\ it is closed (the non-degeneracy of $\omega$ follows directly from that of $g$).
 The simplest, non-flat  example of such a para-K\"ahler manifold may be constructed in an analogous way to the complex projective space:  $\D\mathbb P^n$ is the set of two-dimensional subspaces   of $\D^{n+1}$ which are $J$-stable and have non-degenerate (hence indefinite) metric.  Alternatively, $\D\mathbb P^n$ is the quotient of the quadric $\mathbb Q :=\{ \< \cdot,\cdot\>=1\}$ of $\D^{n+1}$ by the Hopf action $ z \mapsto z \cdot e^{\tau t} $. 

\section{Distinguished submanifolds in para-complex geometry}

\subsection{Para-complex submanifolds} \
It is well known that complex submanifolds in $\C^n$ (or more generally in K\"ahler manifolds) are examples of not only minimal, but even calibrated submanifolds. In the para-complex case, some care must be taken to the definition of a para-complex submanifold: given a submanifold $\s$ of a
 para-complex manifold $(\M,J)$, the assumption that the tangent bundle of $\s$ is stable for $J$ is not sufficient, since the restriction of the eigen spaces $Ker(J\mp Id)$ may not have the same rank. Hence a submanifold $\s$ of $(\M,J)$ will be called \em para-complex \em if the restriction of $J$ to $T\s$ is still an almost para-complex structure, i.e.\ $Ker(J\mp Id)|_{T\s}$ have the same rank.
 It turns out that if $\M$ is in addition para-K\"ahler, its para-complex submanifolds  are  minimal, but fail to be calibrated, being always unstable with respect to the volume form.

\begin{theo} Let $(\M, J, g, \omega)$ a para-K\"ahler manifold and $\s$ a non-degenerate submanifold of $\M$ which is para-complex. Then $\s$ is minimal but unstable.
\end{theo}

\begin{proof}Given  a tangent vector field $X,$ $JX$
is also tangent to $\s$. Moreover, the para-complex structure $J$ is
parallel with respect to the Levi-Civita connection $D$,
 so that $ D_Y JX =J D_Y X.$
On the other hand, since the tangent bundle is $J$-invariant, so
are the normal bundle. It follows that
$$ (D_Y JX)^\perp =(J D_Y X)^\perp=J (D_Y X)^\perp,$$
 i.e.\
 $ h(JX,Y)=Jh(X,Y),$ 
where $h$ denotes the second fundamental form of $\s$, i.e. the two-tensor, valued in the normal bundle, defined by $h(X,Y)= (D_X Y)^\perp$.
 Next, an easy modification of the famous Gram-Schmidt process shows that there exists an
 orthonormal frame $(e_1,\ldots, e_{2k})$ on $\s$ such that $e_{2i}=Je_{2i-1},$ $\forall \ i, \, 1 \leq
i \leq k$ (which proves in particular that the dimension of $\s$ is
even). It follows that $g(e_{2i},e_{2i})=-g(e_{2i-1},e_{2i-1}),$ and we deduce:
\begin{eqnarray*}
2k\vec{H}&=&\sum_{i =1}^{2k} \eps_i h(e_i ,e_i)=\sum_{i =1}^k \eps_{2i-1} h(e_{2i-1}
,e_{2i-1})+\eps_{2i}h(e_{2i}, e_{2i})\\
&=&\sum_{i =1}^k \eps_{2i-1}\Big(h(e_{2i-1} ,e_{2i-1})+h(Je_{2i-1},
Je_{2i-1})\Big)\\
&=&\sum_{i =1}^k \eps_{2i-1}\Big(h(e_{2i-1} ,e_{2i-1})-J^2h(e_{2i-1},
e_{2i-1})\Big)=0. 
\end{eqnarray*}
The unstability of $\s$ comes from the fact that its induced metric is  indefinite: if $X$ is not a null tangent vector, $JX$ is tangent as well and  $g(JX,JX)=-g(X,X)$. It has been proved in \cite{MyB} that a minimal submanifold with indefinite induced metric is always unstable. 
\end{proof}

\subsection{ Lagrangian submanifolds in $\D^n$} 

\begin{lemm} \label{holo-vol} Let $\L$ be a Lagrangian submanifold of $\D^n$ and $(X_1, \ldots, X_n)$ a local tangent frame along $\L$.   
 Then the para-complex number $\det_{\D} (X_1, \ldots, X_n)$ is non null if and only if the induced metric on $\L$ is non-degenerate.  Moreover, if $ (X_1, \ldots, X_n)$ is orthonormal, then  $\det_{\D} (X_1, \ldots, X_n)$ is unit, i.e.\ 
$$\Big| \big\langle \mbox{\em det}_{\D} (X_1, \ldots,X_n) ,\mbox{\em det}_{\D} (X_1, \ldots,X_n)\big\>\Big|=1.$$
Finally, if the metric on $\L$ is non-degenerate, the argument of the para-complex number $\, \det_{\D} (X_1, \ldots, X_n)$ does not depend on the choice of the frame $(X_1, \ldots, X_n)$, but only on the submanifold $\L$. 
\end{lemm}
This lemma allows us to give a satisfactory definition to the concept of \em Lagrangian angle: \em

\begin{defi} Let $\L$ be a non-degenerate, Lagrangian submanifold of $\D^n$ and $(e_1,\cdots,e_n)$ a local orthonormal tangent frame along $\L$. Then the para-complex number
$ \Omega :=\det_{\D} (e_1, \ldots, e_n)$ is called   \em para-holomorphic volume \em of $\L$. 
The argument $\beta$ of $\Omega$ is called \em Lagrangian angle \em of $\L$.
\end{defi}

\begin{proof}[Proof of Lemma \ref{holo-vol}] We denote by $(e_1, \ldots ,e_n)$ the canonical para-Hermitian basis of $(\D^n, \<\<\cdot,\cdot\>\>)$. 
Observe first that given any two vectors  $X$ and $Y$ of $\D^n$, we have
$$ X = \sum_{i=1}^n \<\< X, e_i\>\> e_i $$
and 
$$ \<\< X,Y \>\> = \sum_{k=1}^n \<\< X, e _k\>\>\<\< e_k, Y\>\> = 
 \sum_{k=1}^n \<\< X, e _k\>\> \overline{\<\< Y, e_k\>\>}.$$
Hence, setting
$M:= [ \<\<X_i, e_j\>\>  ]_{1 \leq i,j \leq n}$,  the coefficients  of the first fundamental form (induced metric) are 
\begin{eqnarray*} 
g_{ij} & := & \< X_i, X_j \> \\
 &=&\<\< X_i, X_j \>\>\\ 
&=& \sum_{k=1}^n \<\< X_i, e _k\>\> \overline{\<\< X_j, e_k\>\>}
\end{eqnarray*}
(we have used the Lagrangian assumption to get the second equality).
It follows that $[g_{ij}]_{1 \leq i,j \leq n} = M M^*$, where $M^*$ denotes the conjugate matrix of $M$. It follows that
 $$\mbox{det}_{\R} [g_{ij}]_{1 \leq i,j \leq n} = \mbox{det}_{\D} M \mbox{det}_{\D} M^* = \mbox{det}_{\D} M \overline{\mbox{det}_{\D} M}=\<\mbox{det}_{\D} M , \mbox{det}_{\D} M\>,$$
 hence the induced metric is degenerate if and only if  the para-complex number $\mbox{det}_{\D} M$ is null. Moreover,  if the frame $(X_1, \ldots, X_n)$ is orthonormal, the matrix is $\det_{\R}[g_{ij}]$ is orthogonal, so that
$$  \< \mbox{det}_{\D}M, \mbox{det}_{\D} M \>=\mbox{det}_{\R} [g_{ij}]_{1 \leq i,j \leq n}  =\pm 1.$$
 In order to conclude the proof it suffices to observe that if $(Y_1,\ldots, Y_n)$ is another local frame along $\L$, with 
$Y_i = \sum_{j=1}^n  a_{ij} X_i$, we have
$$ \mbox{det}_{\D} (Y_1, \ldots, Y_n) =\mbox{det}_{\D} [a_{ij}]_{1 \leq i,j \leq n} \mbox{det}_{\D} (X_1, \ldots, X_n).$$
Since the coefficients $a_{ij}$ are real, the two para-complex determinants above have the same argument.
\end{proof}

The proof of the next lemma can be found in \cite{MyB}:
\begin{lemm} \label{tri} Let $\L$ be a non-degenerate, Lagrangian submanifold of $\D^n$. Denote by $D$ the Levi-Civita connection of the induced metric on $\L$.
Then
the  $$ {T}(X,Y,Z):=\<D_X Y,J Z\>$$ is tensorial and
tri-symmetric, i.e.\ $${T}(X,Y,Z)={T}(Y,X,Z)={T}(X,Z,Y).$$
\end{lemm}

\begin{theo}\label{nH=JB}
  Let $\L$ be a Lagrangian submanifold of $\D^n$ with non degenerate induced metric. Then the Lagrangian angle $\beta$ and the mean curvature vector $\vec{H}$ of $\L$ are related by the formula
	$$ n {\vec{H}} = - J \nabla \beta, $$
where $\nabla$ denotes the  gradient operator on $\L$ with respect to the induced metric.
\end{theo}

\begin{proof} Let $(e_1, \ldots  , e_n)$ a local
orthonormal frame of $\L.$ The Lagrangian assumption implies that
it is  also a para-Hermitian frame, i.e.\
$$[\<\<e_j,e_k\>\>]_{1 \leq
j,k \leq n}=diag(\eps_1, \ldots,\eps_n),$$
where $\eps_j = \pm 1.$
 In particular, given any
vector $X$ of $\D^n,$ we have
$$ X = \sum_{j=1}^n \eps_j \<\<e_j,X \>\> e_j.$$
 It is sufficient to prove that
$ \< n\vec{H}, J e_j \> = \< \nabla \be , J e_j\>$, i.e.\ $ \<n\vec{H}, \j e_j \> = d\be (e_j).$ Differentiating the identity
$e^{\tau \be}=\Omega(e_1, \ldots ,e_n)$ with respect to the vector
$e_j$, we have
\begin{eqnarray*}
\tau d\be(e_j) e^{\tau \be}&=& \sum_{k=1}^n \Omega(e_1, \ldots , D_{e_j}e_k,  \ldots , e_n)\\
&=& \sum_{k=1}^n \Omega \left(e_1, \ldots , \sum_{l=1}^n \<\<e_l, D_{e_j}e_k \>\> e_l , \ldots , e_n \right)\\
&=&  \sum_{k=1}^n \Omega(e_1, \ldots , \eps_k \<\<e_k, D_{e_j}e_k\>\> e_k , \ldots , e_n)\\
&=& \sum_{k=1}^n \eps_k \<\<e_k,D_{e_j}e_k \>\> e^{\tau \be},
\end{eqnarray*}
hence 
$$\tau d\be (e_j)=\sum_{k=1}^n \eps_k \<\<e_k,D_{e_j}e_k \>\>.$$
Recalling that $\<\<\cdot,\cdot\>\>=\<\cdot,\cdot\> -\tau \omega=\<\cdot,\cdot\> -\tau \<\cdot,J \cdot\>$,
we have
$$\<\<e_k,D_{e_j}e_k\>\>=\< e_k,D_{e_j}e_k\> -\tau \< e_k,JD_{e_j}e_k \>.$$
Differentiating the relation $\<e_k,e_k\>=\eps_k$ in the direction
$e_j$ yields $ \< e_k,D_{e_j}e_k\>=0,$ so that, taking into
account that $\j e_j$ is a normal vector to $\L$,
$$\<\<e_k,D_{e_j}e_k\>\>=\tau \<Je_k,D_{e_j}e_k\>=\tau T( e_j,e_k,e_k)\>. $$
By Lemma \ref{tri},  we deduce that
$$d\be (e_j)=\sum_{k=1}^n \eps_k \<J e_j, h(e_k,e_k)\>= \< J e_j ,  n\vec{H}\>=-\< e_j, n J\vec{H}\>,$$
which, by the very definition of $\nabla$, is equivalent to the claimed formula. \end{proof}

\begin{coro}\label{constantangle}
  Let $\L$ be a Lagrangian submanifold of $\D^n$ with non degenerate induced metric. Then it is minimal if and only if it has constant Lagrangian angle.
\end{coro}

\begin{rema} \label{rotation} The isometry $e^{\tau \varphi_0}, \, \varphi_0 \in \R,$ transforms a Lagrangian submanifold with para-holomorphic volume $\Omega$ and Lagrangian angle $\be$ into a Lagrangian submanifold with para-holomorphic volume $\Omega e^{\tau \varphi_0}$ and
Lagrangian angle $\be + n \varphi_0$. Hence, in the study of minimal Lagrangian submanifolds, there is no loss of generality in studying Lagrangian submanifolds with vanishing Lagrangian angle, or equivalently, with para-holomorphic volume satisfying $\Re \Omega=0$  or $\Im \Omega=0$ (unlike  the K\"ahler or para-K\"ahler case, these two equations are not equivalent).

\end{rema}

\begin{rema}
It was proved in \cite{Me} (see also \cite{HL2}) that a minimal Lagrangian submanifold with definite induced metric is a volume extremizer in its Lagrangian homology class.
\end{rema}

\section{Minimal Lagrangian submanifolds in $\D^n$}
\subsection{A local characterization of minimal Lagrangian surfaces with indefinite metric} 
\begin{theo}
Let $\L$ be a minimal Lagrangian surface of $\D^2$ with indefinite induced metric. 
Then $\L$ is the product $\ga_1 \times \j \ga_2 \subset P \oplus \j P $, where  $\ga_1$
and $\ga_2$ two planar curves contained in  a non-Lagrangian (and therefore non-complex) null plane $P.$ 
\end{theo}

\begin{proof}
Since the induced metric is assumed to be indefinite, there exist local null coordinates $(u,v)$ on $\L$, i.e.\ such that the induced metric takes the form $F(u,v) du dv$  (see \cite{We}, \cite{MyB}). It follows that, given a local parametrization of $\L$, 
the mean curvature vector $\vec{H}$ is given by the formula $\pa_{uv} f$ (see \cite{Ch}, \cite{MyB}).
Hence the minimality assumption amounts to the partial differential equation  $\pa_{uv} f=0$.
Therefore, the local parametrization $f$ takes the form
$f(u,v)= \ga_1(u) + \tilde{\ga}_2(v),$ where $\ga_1,\tilde{\ga}_2$ are two null curves of $\D^2$ (i.e., 
$ \<\ga_1',\ga_1'\>$ and $\<\tilde{\ga}'_2,\tilde{\ga}'_2\>$ vanish), and

$$\<\ga'_1(u),\tilde{\ga}'_2(v)\> \neq 0, \quad \forall \, (u,v) \in I_1 \times I_2.$$
On the other hand the Lagrangian assumption is:  
$$\omega(\ga'_1(u),\tilde{\ga}'_2(v)) = 0, \quad \forall \, (u,v) \in I_1 \times I_2.$$
The remainder of proof relies on the analysis of the dimension of the two linear spaces $P_1 := Span \{ \ga'_1(u), u \in I_1 \}$ and $P_2 :=Span \{ \tilde{\ga}'_2(v), v \in I_2 \}.$ We first observe that $\dim P_1, \dim P_2 \geq 1$ and that the case $\dim P_1 = \dim P_2 =1$ corresponds to
the trivial case of $\L$ being planar. 
Since the r\^{o}les of $\ga_1$ and $\tilde{\ga}_2$ are symmetric, we may assume without loss of generality, and we do so, that 
$\dim P_1 > 1$. 

Next, the Lagrangian assumption is equivalent to
 $P_2 \subset P_1^{\omega} $ and $P_1 \subset P_2^\omega,$
so $\dim P_2 \leq \dim P_1^{\omega}$ and $\dim P_1 \leq \dim P_2^\omega $.
By the non-degeneracy of $\omega,$ it follows that
$\dim P_1 \leq \dim P_2^\omega = 4 - \dim P_2 \leq 3.$ We claim that in fact $\dim P_1= 2$. To see this,
assume by contradiction that $\dim P_1 = 3$. It follows that $\dim P_2 \leq \dim P_1^{\omega} =1,$ so 	 the curve $\tilde{\ga}_2$ is a straight line, which may be parametrized as follows: $\tilde{\ga}_2(v)=e_0 v, $ where $e_0$ is a null vector of $\D^2.$
Then ${\ga}_1'$ is contained in the intersection of the light cone $\X^4_{2,0}$ with the hyperplane $\{e_0\}^{\omega}.$  Moreover, $e_0 \notin Ker(J \mp Id)$. Hence, after a possible linear, para-complex, isometry of $\D^2$ we may assume that $e_0=(1,0,0,\pm 1)$. Then an easy computation 
	shows that $\X^4_{2,0} \cap \{e_0\}^{\omega}=\Pi_1 \cup \Pi_2$, where $\Pi_1$ and $\Pi_2$ are two null planes. Moreover, one of these  planes, say $\Pi_2$,	is contained in the metric orthogonal of $e_0$. By the non-degeneracy assumption, 
	$$\<\ga'_1(u),\tilde{\ga}'_2(v)\>_2 = 
	v \<\ga'_1(u), e_0\> \neq 0,$$ 
	we deduce that $\ga'_1 \in \Pi_1,$ which implies that $\dim P_1 \leq 2,$ a contradiction.

To conclude, observe  that,
	$\tilde{\ga}_2 \in P_2 \subset P_1^\omega= \j P_1.$ Hence we just need to set $P:=P_1$ and $\ga_2:=  \j \tilde{\ga}_2,$ to get that
	 $\ga_1, \ga_2 \subset P,$ so that $\L$ takes the required expression. \end{proof}

\subsection{Minimal graphs} \
In this section, we look at the minimal Lagrangian graph equation. It is well known that the graph of a closed one-form of a manifold $\M$ is a Lagrangian submanifold of $T^* \M$. In the case of $\M$ being an open subset $U$ of $\R^n$, it is equivalent, locally (or globally $U$ is simply connected), to look at the graph  of the gradient of a map $u \in C^2(U)$, a Lagrangian in 
$T^*U \subset T^* \R^n \simeq \D^n$. In order to write the minimality condition with respect to the para-K\"ahler metric, we consider the immersion
$$\begin{array}{llll} F : &U  &\longrightarrow& \D^n
 \\  & (x_1, \ldots,x_n)  &\longmapsto& \left(x_1 +\tau \frac{\pa u}{\pa x_1}, \ldots ,x_n +\tau \frac{\pa u}{\pa x_n} \right).
\end{array}$$
Since we have
$$ \frac{\pa F}{ \pa x_i} = 
\left(\tau \frac{\pa^2 u}{\pa x_i \pa x_1}, \ldots , 1 + \tau \frac{\pa^2 u}{\pa x_i^2 },\ldots, \tau \frac{\pa^2 u}{\pa x_i \pa x_n} \right),$$ 
Hence using by Theorem \ref{constantangle}, we obtain
$ \beta = \arg \left( det_{\D} (Id + Hess (u)) \right)$, where $Hess$ denotes the matrix made of the second derivatives of $u$.  We note that   the PDE $ \arg \left( det_{\D} (Id + Hess (u)) \right)= const$ may be elliptic or hyperbolic, depending on the causal character of the para-complex number $ det_{\D} (Id + Hess (u))$.\footnote{The corresponding equation is elliptic in  the K\"ahler case (\cite{HL1}) and hyperbolic  in the pseudo-K\"ahler case (\cite{MyB,An3}).}

 In the case $n=2$, this equation takes a remarkable form: 
$$ \arg \left( 1 + det_{\R}(Hess(u)) + \tau \Delta u) \right)= const.$$
If the para-complex number $1 + det_{\R}(Hess(u)) + \tau \Delta u)$ has positive squared norm, then the induced metric on the Lagrangian $F(U)$ is definite. In this case, the equation $\arg \left( 1 + det_{\R}(Hess(u) + \tau \Delta u) \right)=0$ is equivalent, according to Remark \ref{rotation}, to
   the Laplace equation $\Delta u =0$. If $1 + det_{\R}(Hess(u) + \tau \Delta u)$ has negative squared norm,  the induced metric on  $F(U)$ is indefinite, and according to Remark \ref{rotation}, the equation amounts to the hyperbolic Monge-Amp\`ere equation $ det_{\R}(Hess(u))=-1$.  


\subsection{Minimal Equivariant Lagrangian submanifolds in $\D^n$} \label{2.2} \
In this section we describe explicitely some solutions of the minimal equation which are equivariant.
We define the following immersions 
$$\begin{array}{llll} F : & I \times \S^{n-1} &\longrightarrow& \D^n
 \\  & (s,\si)  & \longmapsto & \ga(s) \iota(\si),
\end{array}$$
where $\ga \subset \D^2$ is a regular,  planar curve with $\ga \neq 0$, and $\iota : \S^{n-1} \longrightarrow \R^n$ denotes the canonical embedding of the sphere. In the following, we shall make the following abuse of notation: $\iota(\si)= \si$.
The image of such an immersion is a Lagrangian submanifold of $\D^n$ which is invariant under the action of $SO(n)$ given by $ x+\tau y \mapsto Ax + \tau Ay, $ where $x, y \in \R^n$ and $A \in SO(n)$. It is proved in \cite{Sa} that if $n \geq 3$, any Lagrangian submanifold which is invariant by this action maybe locally parametrized by such an immersion. This is, however, not true if $n=2$.

We now proceed to calculate the para-holomorphic volume of $X$.
Given $\si \in \S^{n-1}$, we introduce a local orthonormal frame $(e_1, \ldots, e_{n-1})$ in a neighbourhood of $\si$,  which is parallel at $\si$. It follows that the vectors
$$ dF(\pa_s)= F_s = \dot{\ga}(s) \si$$ and
$$ dF(e_i) = \ga(s) d \iota (e_i)=\ga(s) e_i, \quad 1 \leq i \leq n-1$$
span the tangent plane (here and in the following, the dot $\, \dot{}\,$ denotes derivative with respect to the variable $s$).
So it is easily seen that the induced metric is degenerate if and only if $\ga$ is null or takes value in the light cone $\<z,z\>=0$
 and that $\Omega$ is,  up to a real constant, $  \dot{\ga}(s) \ga^{n-1}(s).$ 
In particular, the induced metric is definite or Lorentzian, depending on whether $\ga$ and $\dot{\ga}$ have the same causal character, and the immersion $F$ is minimal if and only if 
$\arg ( \dot{\ga} \ga^{n-1} ) $ is constant. According to Remark \ref{rotation}, it is sufficient to look at the two cases
$ \Re  \dot{\ga} \ga^{n-1}=0$ and $\Im  \dot{\ga} \ga^{n-1}=0$. Integrating these equations, 
 we obtain the algebraic equations $\Re \ga^n= C$  and $\Im \ga^n=C$, where $C$ is a real constant. Since the transformation $z \mapsto Jz$ is an anti-isometry of $\D^n$, it preserves minimality. Accordingly, if $\ga$ is a solution of one of these two equations, so is $\tau \ga$. We observe that in the even case, the solutions are \em invariant \em by the axial symmetry $z \mapsto \tau z$, while in the even case, this symmetry \em exchanges \em solutions of $\Re \ga^n= C$  and $\Im \ga^n=C$ respectively.

\subsubsection*{The case $n=2$} \label{torus}
Here the solutions of the first equation corresponds to  round circles $x^2+y^2 = C$. These circles cross four times the light cone $x^2 = y^2$ (i.e.\ the symmetry lines of $z \mapsto \pm \tau z$).  The resulting surface has equation
$$  (x_1-y_2)^2+(y_1+x_2)^2 = (x_1+y_2)^2 +(x_2-y_1)^2 = C.$$
It is a compact torus with four closed lines of null points and with indefinite metric elsewhere.

\smallskip

The solutions of the second equation are the hyperbolas\footnote{These curves do not have constant curvature with respect to the metric $dx^2-dy^2$.} $2 xy = C$. These curves cross once the light cone $x^2=y^2$. The resulting surface is a cylinder with a circle of null points and with definite metric elsewhere.

\subsubsection*{The case $n=3$} \
The first equation takes the form $x^3+3xy^2= C$. The corresponding curve cross twice the the light cone $x^2=y^2$. The resulting surface is a $\S^2 \times \R$ with two spheres $\S^2$  of null points bounding a cylinder with indefinite metric and two unbounded annuli with definite metric.
The second equation $3x^2y+y^3=C$ is equivalent to the first one.

\subsubsection*{The general case} \
Writing
 $ \ga = \pm \tau^\al r e^{\tau \varphi}$
we have
$$\ga^n = \pm \tau^{n \al} r^n e^{\tau n\varphi} = \pm \tau^{n \al} ( r^n\cosh(n\varphi) + \tau r^n \sinh(n \varphi)).$$
We therefore obtain the general solutions
$$ r= \left( \frac{C}{\cosh(n\varphi)} \right)^{1/n} \quad
\mbox{ or }
\quad r= \left( \frac{C}{\sinh(n\varphi)} \right)^{1/n}.$$
In the first case, as $\varphi \to \pm \infty$, the curve $\ga$ tends to the two points $ \frac{ (2C)^{1/n}}{2}(1 , \pm 1)$ of the light cone, that it touches orthogonally. 

In the second case,  as $\varphi \to + \infty$, the curve $\ga$ tends again to a point of the light cone, that it touches orthogonally, while, when $\varphi \to 0$, the curve is asymptotic to a coordinate axis $x=0$ or $y=0$.

\subsection{Normal bundles} \
The \em normal bundle \em $\ove{\s}$ of a submanifold $\s$ of $\R^n$ is the set 
$$\ove{\s}:= \{ (x,\nu) \in \D^n = \R^n \oplus \R^n   \, | \, \,   x \in \s,  \nu \in N_x \s \}.$$
We recall that a submanifold $\si$ is said to be \em austere \em if, for any normal vector field $\nu$, the corresponding curvatures, i.e. the eigenvalues of the quadratic form
$A_\nu := \< h(\cdot,\cdot),\nu\>$, are symmetrically arranged around zero on the real line.

In the complex case $\C^n$, it has been proved in \cite{HL1} that $\overline{\s}$ is a Lagrangian submanifold, and that it is minimal if and only if $\s$ is austere. Since the symplectic structure of $\D^n$ is the same as that of $\C^n$, this fact that $\overline{\s}$ is Lagrangian still holds true here. It turns out that  the minimality criterion for $\overline{\s}$ is the same in the para-complex case:

\begin{theo}
The Lagrangian submanifold $\overline{\s}$ is minimal with respect to the neutral metric $\<\cdot,\cdot\>$ if and only if $\s$ is austere.
\end{theo}

\begin{proof} Let $x $ be a point of $\s$ and $(\nu_1,\ldots,\nu_{n-p})$ a local orthonormal frame of the normal space $N\s$, defined in a neighbourhood of $x.$ There exists a local orthornormal frame $(e_1, \ldots, e_p)$ in a neighbourhood of $x$ which is principal with respect to $\nu_1$, i.e.\ it diagonalizes $A_{\nu_1}.$ We denote by $\ka_1, \ldots , \ka_p$ the corresponding principal curvatures, i.e.\ $A_{\nu_1} e_i = \ka_i e_i, \, \forall \ i, 1 \leq i \leq p.$ Observe that if $t \in \R,$ 
$A_{t \nu_1} = t A_{\nu_1}$.

We are going to calculate the Lagrangian angle of $\overline{\s}$ at a point $(x,\nu)$. Without loss of generality, be may assume that $\nu=t \nu_1, \, t \in \R$.

We set
$$ E_i(x, \nu):= (e_i, -t \ka_i e_i)  \,\,\,\, \forall i, \, \, 1 \leq i \leq p$$
and
$$ E_{p+j}(x, \nu):=( 0, \nu_j)   \,\,\,\, \forall j, \, \, 1 \leq j \leq n-p.$$
We therefore get a local tangent frame $(E_1, \ldots, E_n)$ along $\overline{\s}$ in a neighbourdhood of $(x,\nu)$. Next, we calculate
$$ det_{\D} (E_1, \ldots,  E_n)=\tau^{n-p} \prod_{i=1}^p (1 -\tau t \ka_i).$$
Using Theorem \ref{constantangle}, we get
$$ \beta(x,\nu)= \arg \left( \tau^{n-p} \prod_{i=1}^p (1 -\tau t \ka_i) \right).$$
Proceeding like in \cite{HL1}, it is easily seen that this number does not depend on $t$ if and only if the principal curvatures $\ka_i$ are symmetrically arranged around zero on the real line, i.e.\ $\s$ is austere.
\end{proof}

\section{Equivariant self-similar Lagrangian submanifolds in $\D^n$} \
In this section, we describe some self-similar Lagrangian solutions of the Mean Curvature Flow (MCF).
The simplest and most important example of a self-similar flow is when the evolution
is a homothety. Such a self-similar submanifold $F$ with mean curvature vector $\vec{H}$
satisfies the following non-linear, elliptic system:
\begin{equation} \label{SS} \vec{H} + \lambda F^\perp = 0,\end{equation}
where $F^\perp$ denotes the projection of the position vector $F$ onto the normal space of the submanifold, and $\lambda$ is a real constant. If $\lambda$ is strictly positive constant, the submanifold shrinks in finite
time to a single point under the action of the MCF, its shape remaining
unchanged. If $\lambda$ is negative, the submanifold will expand, its shape
again remaining the same. In the case of vanishing $\lambda$, we recover the case of minimal submanifolds, which of course are the stationary points of the MCF. We point out that in the para-K\"ahler setting, unlike the K\"ahler case, the shrinking and expanding case are roughly equivalent, since a change of sign in the metric $\<\cdot,\cdot\> \longmapsto -\<\cdot,\cdot\>$ (which can be explicitly performed by  applying the point transformation $z \mapsto Jz$) yields  a change of sign in the mean curvature vector $\vec{H} \longmapsto -\vec{H}$ and leaves unchanged the term $F^\perp$. Hence, without loss of generality, we may assume that $\la \ge 0$.

As in Section \ref{2.2}, we consider the immersions
$$\begin{array}{llll} F : & I \times \S^{n-1} &\longrightarrow& \D^n
 \\  & (s,\si)  &\longmapsto& \ga(s) \iota(\si),
\end{array}$$
where $\ga \subset \D^2$ is a regular,  planar curve with $\ga \neq 0.$ 
We recall that $(e_1, \ldots, e_{n-1})$ is a local orthonormal frame. In particular,
 the vectors
$$ J F_s=\tau \dot{\ga} \si  \, \, \mbox{ and } \, \, \, JdF(e_i)=\tau \ga e_i,  \, \, \,  1 \leq i \leq n-1$$
span the normal space of $F$.
A straightforward \, computation shows that   $\<\vec{H}, JF(e_i)\>$ and $\<F,  JF(e_i)\>$ vanish $\forall \ i, \, \, 1 \leq i \leq n-1$. Hence the  self-similar equation  (\ref{SS}) is equivalent to the scalar equation
\begin{equation} \label{SS-equi} \<\vec{H},J F_s \>  + \lambda \< F^\perp, JF_s \>=0.\end{equation}
Next, we calculate
\begin{eqnarray*} \<\vec{H} , JF_s\> &=& \frac{\< \ddot{\ga}, J \dot{\ga} \>}{\< \dot{\ga},\dot{\ga}\>} - (n-1) \frac{\<\ga,J \dot{\ga}\>}{\<\ga,\ga\>}.  
\end{eqnarray*}
We also have
$$\<F,JF_s\>=\< \ga \si,  J \dot{\ga} \si\>=\<\ga, J \dot{\ga}\>,$$
hence Equation (\ref{SS-equi}) becomes 
\begin{equation}\label{SS-curva} \frac{\<\ddot{\ga}, J \dot{\ga}\>}{\<\dot{\ga},\dot{\ga}\>} 
+ \left(-  \frac{ n-1}{\<\ga,\ga\>}  + \lambda \right)\< \ga, J \dot{\ga}\>.
 \end{equation}
We now assume that $\ga$ is parametrized by ``arclength", i.e.\ $ \< \dot{\ga},\dot{\ga}\> =\pm 1:= \eps'$. 
It follows that $\nu:=J \dot{\ga}$ is a unit normal vector and that the acceleration vector $\ddot{\ga}$ is normal to the curve, hence collinear to $\nu$. 

It follows that there exists $(p',q') \in \{0,1\}^2$ and
 $\te: I \longrightarrow \R$ such that
$$\dot{\ga}(s)= (-1)^{p'} \tau^{q'}  \big(\cosh(\te(s)) + \tau \sinh(\te(s))\big).$$ 
Differentiating, we get
$$ \<\ddot{\ga}, J\ga'\> =(-1)^{q'+1} \dot{\te}=-\eps' \dot{\te}.$$
In order so study Equation (\ref{SS-curva}),  we use ``polar" coordinates:
$$\ga(s)= (-1)^p \tau^q r(s)\big(\cosh(\varphi(s)) + \tau \sinh(\varphi(s))\big)$$
where $(p,q) \in \{0,1\}^2$, \,  $r(s): I \longrightarrow (0,\infty)$ and $\varphi(s): I \to \R$. 
Observe that $ q = \displaystyle\frac{1-\eps}{2}$. Moreover, the induced metric is definite if $q=q'$ and Lorentzian if $q \neq q'$.
Hence Equation $(\ref{SS-curva})$ becomes
\begin{equation}\label{SS'}- \dot{\te} + \left( - \eps \frac{ n-1}{r^2}  + \lambda \right)\< \ga, \nu\> =0.\end{equation}
For the remainder of the analysis of the equation it is convenient to deal separately with the definite and Lorentzian cases:

\subsubsection*{The definite case} \
Differentiating $\ga$ and setting $\al:= \te -\varphi$, we have the following ``compatibility equation":
$$   ( \dot{r} ,  \dot{\varphi} )  =   \eta \big( \cosh \al ,   \frac{1}{r} \sinh \al \big) $$
where we have set for brevity $\eta:=(-1)^{p+p'}$.
On the other hand, a calculation gives
$$\<\ga, \nu\> =   \eta \, \eps \, r \sinh \al .$$
So finally, putting together the compatibility equations and   Equation $(\ref{SS-curva})$, we get a $3 \times 3$ system
$$\left\{ \begin{array}{ccl}  \dot{r} &  =&  \eta   \cosh \al  \\ \dot{\varphi} &=& \eta \frac{1}{r} \sinh \al  \\
\dot{\te} &=&  \eta \left( - \frac{n-1}{r} +  \la' r \right)\sinh \al,
\end{array} \right.   $$
 where we have set for convenience $\la':=\eps \la$.
Observe that the integral curves of the system do not depend on  $\eta$, so without loss of generality we may assume that $\eta=1$. Moreover, recalling that $\al=\te-\varphi$, we are left with the $2 \times 2$ system:
$$\left\{ \begin{array}{ccl}  \dot{r} &  =&     \cosh \al  \\ \dot{\al} &=&    \left( - \frac{n}{r} +  \la' r \right)\sinh \al \end{array} \right.   $$
The system enjoys a first integral: $E(r, \al) =   r^n \exp(- \la' r^2/2) \sinh \al.$ It follows that the projection on the plane $\{ (r, \al) \}$ of the integral curves $(r, \al, \varphi)$ have two ends, one with $ (r, \al) \longrightarrow (  0,\pm \infty)$ and the other with $(r, \al ) \longrightarrow (+ \infty , \pm \infty)$ if $\la'>0$ or $(r, \al ) \longrightarrow (0 , \pm \infty)$ if $\la' \leq 0.$ In order  to understand the asymptotic behaviour of $\ga$, we need to study the variable $\varphi$.
Using the fact that
\begin{eqnarray*} \frac{d\varphi}{dr} =
 \frac{ \sinh \al}{r \cosh \al}=
\frac{1}{r} \left(\frac{r^{2n} \exp(-\la' r^2)}{E^2}+1\right)^{-1/2},
\end{eqnarray*}
we obtain that
\begin{eqnarray*} \varphi(s) = \varphi(0) + \int_0^s \dot{\varphi}(\sigma)d\sigma 
= \varphi(0)+ \int_{r_0}^r \frac{1}{\rho} \left( \frac{\rho^{2n} \exp(-\la' \rho^2)}{E^2}+1 \right)^{-1/2}d \rho. 
\end{eqnarray*}
We first see that $\varphi\sim_{r \sim 0} \log(r)+C$, which implies that  the curve touches orthogonaly the light cone as $r \to 0$. 
On the other hand, as $r \to +\infty$, the situation depends on $\la'$: if the case $\la'>0$, we still have $\varphi \sim_{r \sim +\infty} \log(r)+C$, which implies that as $r \to +\infty$, the curve $\ga$ becomes asymptotically parallel to a branch of the light cone, while if $\la' \leq 0$, $\varphi$ converges to a constant, which means that the curve $\ga$ is asymptotic to a straight line.

\subsubsection*{The Lorentzian case}\
We proceed exactly as in the previous case. Setting $\al:= \te- \varphi$ and $\la':=\eps \la$, we are left with the  $2 \times 2$ system:
$$\left\{ \begin{array}{ccl}  \dot{r}&  =&     \sinh \al  \\ \dot{\al} &=&   \left( - \frac{n}{r} +  \la' r \right)\cosh \al,
\end{array} \right. $$
which enjoys a first integral: $E(r, \al) = 
  r^n \exp(- \la' r^2/2) \cosh \al. $
In order to describe the global behaviour of $\ga$, we come back to the sudy of $\varphi$. Since  we have
\begin{eqnarray*} \frac{d\varphi}{dr} &=&
  \frac{1}{r} \left(1-\frac{r^{2n} \exp(-\la' r^2)}{E^2}\right)^{-1/2}
\end{eqnarray*}
 we obtain
\begin{eqnarray*} \varphi(s) =\varphi(0)+ \int_{r_0}^r \frac{1}{\rho} 
\left(1-\frac{\rho^{2n} \exp(-\la' \rho^2)}{E^2}\right)^{-1/2}d \rho .
 \end{eqnarray*}
In the case $\la' \leq 0$,  the variable $r$ is bounded on  the integral curves $(r, \al, \varphi)$, which have  two ends with $ (r, \al) \longrightarrow (  0 ,\pm \infty)$. We still have  $\varphi\sim_{r \sim 0} \log(r)+C$ so again  the curve $\ga$ touches orthogonally the light cone as $r \to 0$ (in the case $\la'=0$ and $n=2$, we recover the example of the torus found in Section \ref{torus}).

If   $\la' >0$, the phase diagram is a more complicated: it has a critical point $(r_0, \al_0)= ( \sqrt{\frac{n}{\la'}}, 0)$, which corresponds to $\ga$ being a hyperbola (constant curvature curve): if $\eps=-1$, we have  $ \ga(s)=\pm  r_0 ( \cosh(s/r_0), \sinh(s/r_0))$  which is ``self-expanding" since here $\la<0$, and if $\eps=1$, we have
$\ga(s)=\pm  r_0 ( \sinh(s/r_0), \cosh(s/r_0))$, (which is ``self-shrinking"  since $\la>0$). 

The properties of the other solutions depend on the energy level $E_0:=\left(\frac{n}{\la'}\right)^{n/2} \exp(- n^2/2)$:

\begin{itemize}
\item[---] Curves with $E<E_0$ and $r < r_0$. They are symmetric with respect to the $r$ axis and enjoy two ends with $r \to 0$. There are similar to the Lorentzian case with $\la'>0$;

\item[---] Curves with $E<E_0$ and $r >r_0$. They are also symmetric with respect to the $r$ axis and enjoy two ends with $r \to + \infty$ which are asymptotically parallel to the light cone;

\item[---] Curves with $E \geq E_0$; These curves enjoys one end with $r \to 0$, with the curve touching orthogonally the light cone, and another one with $r \to +\infty$, asymptotically parallel to the light cone. 
\end{itemize}

\section*{Appendix: para-complex and almost-para-complex manifolds}\
The purpose of this Appendix is to prove the following statement:

\smallskip

\em An almost para-complex structure $J$ comes from a para-complex structure if and only if its para-Nijenhuis tensor $N^J$ vanishes. \em

\smallskip

\begin{proof} Given a almost para-complex structure $J$, we denote by ${\cal D}^+$ and ${\cal D}^-$ the two corresponding eigen distributions, i.e.\ ${\cal D}^{\pm}:= Ker ( J \mp Id)$. 
 The following lemma, whose  proof we  leave to the reader, which expresses the para-Nijenhuis tensor in terms of the decomposition: 

\begin{lemm} \label{Nijenhuis}
Take two vector fields $X_1$ and $X_2$ on $\M$ and write the decomposition $X_1 =U_1 +V_1$ and
$X_2 =U_2+V_2$ in terms of $T\M={\cal D}^+ \oplus {\cal D}^-$. Then we have
$$N^J (X_1 , X_2 ) = 2 \Big([U_1 , U_2 ] - J [U_1 , U_2 ] + [V_1 , V_2 ] + J [V_1 , V_2 ]\Big).$$
\end{lemm}

We assume first that $N^J$ vanishes. By Lemma \ref{Nijenhuis},  $X_1 \in {\cal D}^+$ (hence $V_1 = 0$) implies that $[U_1, U_2] = J[U_1, U_2]$, i.e. $[U_1, U_2] \in {\cal D}^+$, i.e.${\cal  D}^+$ is integrable. Analogously, if $X_2 \in {\cal D}^-$, we get that $ [V_1,V_2] \in {\cal D}^-,$ i.e. $ {\cal D}^-$ is integrable.

Hence, in the neighbourhood of any point we can find local coordinates $(u_1, \cdots , u_n, v_1, \cdots, v_n)$ on $\M$ such that the integral submanifolds of
${\cal D}^+$ (resp.\ $ {\cal D}^-$) are $\{v_i = const, 1 \leq i \leq n \}$ (resp.\ $\{u_i = const, 1 \leq
i \leq n\}$). In particular, we have $J \, \pa u_i = \pa u_i$ and $J \, \pa v_i = − \pa v_i$. We define another local system of coordinate we setting
$x_i := u_i+v_i, \, y_i := u_i-v_i $. We 
claim that the collection of all such system of coordinates define an atlas of para-complex coordinates. 
To see this, we need to prove that the transition functions satisfy the para-Cauchy-Riemann equations. Let $(x_1, \ldots, x_n, y_1, \ldots, y_n)$ and $(x_1', \ldots, x_n', y_1' , \ldots, y_n' )$ two local system of coordinates constructed as before.
Let $(u_1, \ldots , u_n, v_1, \ldots, v_n)$ and $(u_1', \ldots, u_n', v_1' , \ldots , v_n' )$ be the associated ``null" coordinates (between quotes because here there is no metric). Since
$$ \{ u_i =const.,  \, \,  1 \leq i \leq n \} = \{ u_i' =const., \, \, \,   1 \leq i \leq n \} $$
and
$$ \{ v_i =const., \, \, \  1 \leq i \leq n \} = \{ v_i' =const., \, \, \  1 \leq i \leq n \}, $$
we have
$$ \frac{\pa u_i'}{\pa v_j}  =  \frac{\pa v_i'}{\pa u_j}=0, \, \, \forall \ i, j, \, \, \, 1 \leq i, j \leq n,$$
which is equivalent to the para-Cauchy-Riemann equations
$$ \frac{\pa x_i'}{\pa x_j}  =  \frac{\pa y_i'}{\pa y_j} \quad \mbox{ and }  \quad \frac{\pa x_i'}{\pa y_j}  =  \frac{\pa y_i'}{\pa x_j}, \, \, \forall \ i, j, \, \, \, 1 \leq i,j \leq n.$$
Conversely, let $\M$ be a manifold equipped with a para-complex structure, i.e.\ an atlas whose transition functions satisfy the para-Cauchy-Riemann equations. This implies the global existence of two foliations: locally, they are defined by  $\{u_i = const,1 \leq i \leq n\}$ and $\{v_i = const,1 \leq i \leq n\}$ and the para-Cauchy-Riemann equations precisely say that the definition is independent of a particular choice of coordinates, and may therefore be extended globally. Denoting by ${\cal D}^+$ and ${\cal D}^-$ the two distributions tangent to this foliations, which are the eigen-spaces of $J$. By Frobenius theorem, $[U_1, U_2] \in {\cal D}^+$ and $[V_1,V_2] \in {\cal D}^-$. By Lemma \ref{Nijenhuis}, $N^J(X_1,X_2)$ vanishes. \end{proof}

\end{document}